\documentclass[11pt]{amsart}

\usepackage{latexsym}
\usepackage{amsfonts}

\newenvironment{thm}{\subsection{}{\textbf {Theorem.}}\em}{}
\newenvironment{prop}{\subsection{}{\textbf {Proposition.}}\em}{}
\newenvironment{cor}{\subsection{}{\textbf {Corollary.}}\em}{}
\newenvironment{lem}{\subsection{}{\textbf {Lemma.}}\em}{}

\newenvironment{eg}{\subsection{}{\textbf {Example.}}}{\smallskip}

\newcommand\fA{\ensuremath{\mathfrak A}}

\newcommand\cA{\ensuremath{\mathcal A}}
\newcommand\cB{\ensuremath{\mathcal B}}

\newcommand\cD{\ensuremath{\mathcal D}}
\newcommand\cE{\ensuremath{\mathcal E}}
\newcommand\cF{\ensuremath{\mathcal F}}

\newcommand\cH{\ensuremath{\mathcal H}}

\newcommand\cJ{\ensuremath{\mathcal J}}
\newcommand\cK{\ensuremath{\mathcal K}}
\newcommand\cL{\ensuremath{\mathcal L}}
\newcommand\cM{\ensuremath{\mathcal M}}

\newcommand\cQ{\ensuremath{\mathcal Q}}

\newcommand\cS{\ensuremath{\mathcal S}}
\newcommand\cT{\ensuremath{\mathcal T}}

\newcommand\cZ{\ensuremath{\mathcal Z}}

\newcommand\bbC{\ensuremath{\mathbb C}}

\newcommand\bbM{\ensuremath{\mathbb M}}
\newcommand\bbN{\ensuremath{\mathbb N}}

\newcommand\hilb{\ensuremath{\mathcal H}}

\newcommand\ol{\ensuremath{\overline}}

\newcommand\eop{{{\hfil \ensuremath \Box}}}
\newcommand\eps{\ensuremath {\varepsilon}}
\newcommand\norm{\ensuremath {\Vert}}

\usepackage{mathrsfs}
\usepackage[colorlinks, citecolor=blue, linkcolor=red, pdfstartview=FitB]{hyperref}
\usepackage{amssymb, amsmath}
\usepackage{bm}

\renewcommand{\phi}{\varphi}
\renewcommand{\rho}{\varrho}

\begin{document}
\title{Residual finite dimensionality and representations of amenable operator algebras}

\thanks{The first author was partially supported by an FQRNT postdoctoral fellowship and an NSERC Discovery Grant.  The second author was partially supported by an NSERC Discovery grant.}
\author
	[R. Clou\^{a}tre]{{Rapha\"el Clou\^{a}tre}}
\address
	{Department of  Mathematics\\
	University of Manitoba\\
	Winnipeg, Manitoba \\
	Canada  \ \ \ R3T 2N2}
\email{raphael.clouatre@umanitoba.ca}
\author
	[L.W. Marcoux]{{Laurent W.~Marcoux}}
\address
	{Department of Pure Mathematics\\
	University of Waterloo\\
	Waterloo, Ontario \\
	Canada  \ \ \ N2L 3G1}
\email{LWMarcoux@uwaterloo.ca}

\begin{abstract}
We consider a version of a famous open problem formulated by Kadison, asking whether bounded representations of operator algebras are automatically completely bounded. We investigate this question in the context of amenable operator algebras, and we provide an affirmative answer for representations whose range is residually finite-dimen-sional. Furthermore, we show that weak-${}^*$ closed, amenable, residually finite-dimensional operator algebras are similar to $C^*$-algebras, and in particular have the property that all their bounded representations are completely bounded. We prove our results for operator algebras having the so-called total reduction property, which is known to be weaker than amenability.
\end{abstract}

\maketitle
\markboth{\textsc{R. Clou\^atre and L.W. Marcoux}}{\textsc{RFD and representations of amenable operator algebras}}

 
\section{Introduction} \label{sec01}

%


Let $B(\cH)$ denote the $C^*$-algebra of bounded linear operators on some complex Hilbert space $\cH$. Given a Banach algebra $\cA$, we shall refer to a continuous algebra homomorphism $\theta: \cA \to B(\cH)$ as a \emph{representation} of $\cA$.  In particular, representations of $C^*$-algebras are not assumed to be self-adjoint. One of the most  intriguing questions in the theory of $C^*$-algebras is the following, which stems from a 1955 paper of R.V.~Kadison~\cite{Kad1955}. 
\bigskip

\noindent{\textbf{Kadison's similarity problem.}} Let $\fA$ be a $C^*$-algebra and let $\theta:\fA\to B(\cH)$ be a representation. Does there exist an invertible operator $X \in B(\hilb)$ so that the representation $\theta_X: \fA \to B(\hilb)$ defined by
\[
\theta_X(a) = X^{-1} \theta(a) X, \quad a \in \fA
\]
is a ${}^*$-representation? 

\bigskip

We say that $\fA$ has \emph{Kadison's similarity property} if the problem above has an affirmative answer. It is still unknown to this day whether every $C^*$-algebra has this property, although some deep partial results have emerged throughout the years. For instance, E.~Christensen~\cite{Chr1981} has solved the problem in the affirmative for amenable $C^*$-algebras (we note here that due to sophisticated results of  A.~Connes~\cite{Con1978} and U.~Haagerup~\cite{Haa1983}, the classes of amenable and nuclear $C^*$-algebras are known to coincide). It was shown by Haagerup~\cite{Haa1983s} that Kadison's similarity problem has an affirmative answer for every $C^*$-algebra, provided that one restricts one's attention to only those representations which admit a finite cyclic set of vectors. Moreover, it is known~\cite{Haa1983s} that a representation of a $C^*$-algebra is similar to a ${}^*$-representation if and only if it is completely bounded.  This allows one to reformulate Kadison's similarity problem so that it makes sense in a non self-adjoint context. The following problem appeared in~\cite{Pis2001b}.  

\bigskip

\noindent{\textbf{The generalised similarity problem.}} Let $\cA$ be an operator algebra and let $\theta:\cA\to B(\cH)$ be a representation. Is $\theta$ necessarily completely bounded?
\bigskip

Following~\cite{Pis2001b}, we say that the $\cA$ has the \emph{SP property} if the problem above has an affirmative answer. In producing an example of a polynomially bounded operator which is not similar to a contraction --  thereby answering a long-standing open question (i.e. the Halmos Problem) as to the existence of such an operator -- Pisier demonstrated that the classical disc algebra is an example of an operator algebra without the SP property \cite{Pis1997}. On the other hand, a straightforward verification shows that if $\cA\subset B(\cH)$ has the SP property, then so does $X\cA X^{-1}$ for any invertible operator $X\in B(\cH)$. In particular, it follows from Christensen's result mentioned above that any operator algebra that is similar to an amenable $C^*$-algebra has the SP property. 

In the 1980's there arose the question of determining which operator algebras are similar to amenable $C^*$-algebras. Since amenability is preserved under Banach algebra isomorphisms -- of which similarity of operator algebras is an example -- an obvious restriction on such operator algebras is that they be amenable. It was conjectured by a number of people in the Banach algebra community that this was in fact the only obstruction, and that any amenable operator algebra was similar to a $C^*$-algebra. This conjecture was verified in some special cases. Indeed, it was shown to hold for uniform algebras by M.V. {\v S}e{\u \i}nberg \cite{Sei1977}, for abelian amenable subalgebras of finite von Neumann algebras by Y. Choi \cite{Cho2013}, and very recently for arbitrary abelian amenable operator algebras by the second author and A.~Popov~\cite{MP2015}. 

However, other recent developments have shown the conjecture to be incorrect in general,  and an example of a non-abelian, non-separable, amenable operator algebra which fails to be similar to a $C^*$-algebra was constructed in~\cite{CFO2014} by Y. Choi, I. Farah and N. Ozawa. It is not currently known whether a separable, amenable operator algebra must always be similar to a $C^*$-algebra. Nevertheless, this counterexample shows that a positive solution to the generalised similarity problem for amenable operator algebras cannot be achieved through this approach by relying on Christensen's result.  Interestingly, the counterexample does have the SP property, as shown in the companion paper \cite{CM2016id}. Thus, the question of whether every amenable operator algebra has the SP property remains unanswered and it motivated much of our efforts. In fact, we will be interested in a slightly more general question.

An illuminating insight into Kadison's similarity problem and the generalised similarity problem was offered by J.A.~Gifford in his PhD thesis~\cite{Gif1997}. Therein, he obtained a complete characterization of the $C^*$-algebras having Kadison's similarity property as those having a remarkably well behaved lattice of invariant subspaces, a feature he called the \emph{total reduction property} (the precise definition is given in Section \ref{secPr}). He also observed that the total reduction property is strictly weaker than amenability (we point out that the result of \cite{MP2015} mentioned above was in fact proved under this weaker assumption).

We can now state the basic question which we wish to address in this paper.

\bigskip

\noindent{\textbf{Question.}} Let $\cA$ be a norm-closed  operator algebra with the total reduction property.  Is every representation $\theta: \cA \to B(\hilb)$ completely bounded?

\bigskip

The reader will notice that this question is a special case of the  generalised similarity problem. Based on the discussion above, we believe it to be a meaningful step towards a complete understanding of the latter.

We now describe the organization of the paper, and state our main results. Section~\ref{secPr} deals with preliminaries: we give precise definitions of the concepts we require throughout, gather relevant results from the literature and prove some basic facts in preparation for later work. In Section~\ref{secRed}, we explain how the general problem we are trying to solve can be reduced to one involving matrix algebras, at least in the amenable case. 
%
%
%
%
Consequently, in trying to determine whether every amenable operator algebra has the SP property, it is no loss of generality to focus on representations whose \emph{domains} are residually finite-dimensional algebras.  Motivated by this observation, in Section~\ref{sec02} we examine the structure of subalgebras of products of matrix algebras that possess the total reduction property. This information is then leveraged to prove our first main result (see Theorem \ref{thm2.13} and Corollary \ref{cor2.14}).

\begin{thm}\label{thmmain1}
Let $(k(\lambda))_\lambda$ be a net of positive integers and suppose that $\cA \subset \prod_\lambda \bbM_{k(\lambda)}$ is a subalgebra with the total reduction property.  If $\cA$ is closed in the weak-${}^*$ topology, then $\cA$ is similar to a $C^*$-algebra and in particular it has the SP property.
\end{thm}

Although the condition of being weak-$*$ closed restricts the range of applicability of the previous result, we view it as noteworthy partial step towards clarifying the general situation. Finally, Section~\ref{sec04} is devoted to the study of representations whose \emph{range} is residually finite-dimensional, and our second main result is proved therein (see Corollary \ref{cor4.15}). Roughly speaking, it says that one has uniform control on the completely bounded norm of finite-dimensional representations in the presence of the total reduction property.

\begin{thm}\label{thm4.13}
Let $(k(\lambda))_\lambda$ be a net of positive integers, let $\cA$ be an operator algebra with the total reduction property and let $\theta:\cA\to \prod_\lambda \bbM_{k(\lambda)}$ be a representation. Then, $\theta$ is completely bounded.
\end{thm}

The previous result fails without the total reduction property (Example \ref{E:trprfd}).

	

\section{Preliminaries and background material}\label{secPr}

\subsection{Operator algebras and completely bounded maps}

Throughout, by an \emph{operator algebra} we mean a subalgebra of some $B(\cH)$ which is closed in the norm topology.  

If $\cA\subset B(\cH)$ is an operator algebra, then for each integer $n \ge 1$ the algebra $\bbM_n(\cA)$ inherits a norm when viewed as a subalgebra of the $C^*$-algebra $B(\hilb^{(n)})$. If $\cB$ is another operator algebra, then a linear map $\phi: \cA \to \cB$ induces a sequence of maps 
\[
\phi^{(n)}: \bbM_n(\cA) \to \bbM_n(\cB),\quad n \ge 1
\]
by setting 
\[
\phi^{(n)}( [a_{i,j}] )= [ \phi(a_{i,j})]
\]
for all $[a_{i,j}] \in \bbM_n(\cA)$.   The map $\phi$ is \emph{completely bounded} if the quantity
\[
\norm \phi \norm_{cb} = \sup_{n \ge 1} \norm \phi^{(n)} \norm
\]
 is finite.  The map $\phi$ is said to be \emph{completely contractive} (respectively, \emph{completely isometric}) if each $\phi^{(n)}$ is contractive (respectively, isometric). We refer the reader to \cite{Pau2002} for a thorough exposition of the theory of completely bounded maps and of operator algebras.

\subsection{Amenability and the total reduction property}

We recall the notion of amenability of a Banach algebra, which was introduced by Johnson in \cite{Joh1972}. Let $\cA$ be a Banach algebra and let $X$ be a Banach $\cA$-bimodule.
If $X^*$ denotes the dual space of $X$, then $X^*$ also carries the structure of an $\cA$-bimodule under the dual actions  defined by 
\[
(a x^*) (x) = x^* ( xa) \ \ \ \ \ \ \ \mbox{ and } \ \ \ \ \ \ \ \ \ (x^* a) (x) = x^* (ax )
 \]
for all $a \in \cA$, $x^* \in X^*$ and $x \in X$.   When it is equipped with this precise module action inherited from that of $\cA$ on $X$, we say that $X^*$ is a \emph{dual Banach $\cA$-bimodule}.

A \emph{derivation} $\delta: \cA \to X$ is a continuous linear map which satisfies 
\[
\delta (a b) = \delta(a) b + a \delta (b)
\]
for all $a, b \in \cA$.  The derivation is \emph{inner} if there exists a fixed element $z \in X$ so that 
\[
\delta (a) = a z - z a
\]
for all $a \in \cA$.  Finally, we say that $\cA$ is \emph{amenable} if every derivation of $\cA$ into a dual Banach $\cA$-bimodule is inner.

As mentioned in the introduction, we will be mostly concerned with a notion, introduced in the thesis of J.A.~Gifford \cite{Gif1997},  that is weaker than amenability.  An operator algebra $\cA $ is said to have the \emph{total reduction property} if, whenever $\theta: \cA \to B(\cH)$ is a representation and $M \subset \hilb$ is a closed $\theta(\cA)$-invariant subspace, there exists another closed $\theta(\cA)$-invariant subspace $N$ which is a topological complement of $M$, in the sense that $\hilb = M + N$ and $M \cap N  = \{ 0\}$. Equivalently, any closed $\theta(\cA)$-invariant subspace is the range of some bounded idempotent lying in the commutant $\theta(\cA)'$.

 More precise information about these idempotents is available \cite[Proposition 2.2.13]{Gif1997}. Indeed, for each $t\geq 0$ there is a positive constant $C(t)$ with the property that whenever $\theta:\cA\to B(\cH)$ is a representation with $\|\theta\|\leq t$ and $M\subset \cH$ is a closed $\theta(\cA)$-invariant subspace, there is an idempotent $E\in \theta(\cA)'$ with $E\cH=M$ and $\|E\|\leq C(t)$.
For each $t\geq 0$, we denote by $\kappa_\cA(t)$ the minimum of all positive constants $C(t)$ satisfying this condition. Clearly, $\kappa_\cA$ is an increasing function. We single out another one of its basic properties.
 
\begin{lem}\label{lemkapparep}
 Let $\cA$ be an operator algebra with the total reduction property and let $\rho$ be a representation of $\cA$ such that $\varrho(\cA)$ is closed. Then, $\varrho(\cA)$ has the total reduction property. Moreover, if $\theta$  is a representation of $\rho(\cA)$, then
\[
\kappa_{\varrho(\cA)}(\| \theta \|)\leq \kappa_{\cA}(\|\theta \circ \varrho\|).
\]
In particular, for every $t>0$ we see that
\[
\kappa_{\varrho(\cA)}(t)\leq \kappa_{\cA}(t\| \varrho\|).
\]
\end{lem}
\begin{proof}
The fact that $\rho(\cA)$ has the total reduction property is simply \cite[Proposition 3.3.1]{Gif1997}. 

Next, let $\theta:\varrho(\cA)\to B(\cH)$ be a representation and let $M\subset \cH$ be a closed $\theta(\rho(\cA))$-invariant subspace. Since $\cA$ has the total reduction property, there is an idempotent $E\in \theta(\rho(\cA))'$ such that $E\cH=M$ and $\|E\|\leq \kappa_\cA(\|\theta\circ \rho\|)$. We conclude that
\[
\kappa_{\rho(\cA)}(\|\theta\|)\leq \kappa_\cA(\|\theta\circ \rho\|)\leq \kappa_\cA(\|\theta\| \|\rho\|).\]
\end{proof}
 
We mention in passing that if, in the above result, we do not assume that $\rho(\cA)$ is closed, then a similar argument shows that $\overline{\rho(\cA)}$ has the total reduction property, and that $\kappa_{\overline{\rho(\cA)}}(t) \le \kappa_\cA(t \| \rho \|)$ for all $t \geq  0$.

\smallskip

It is relevant to point out that amenability implies the total reduction property \cite[Proposition 2.3.2]{Gif1997}. However, the latter is strictly weaker than amenability:  if $\cH$ is an infinite-dimensional Hilbert space, then $B(\cH)$ has the total reduction property \cite[Corollary~2.4.7 ]{Gif1997} but it is not amenable since it is not nuclear \cite{Sza1981},\cite{Con1978}.

Before proceeding, we gather here several facts about operator algebras with the total reduction property that we require numerous times in the sequel. First, we consider ideals.

\begin{thm}\label{thmideal}
Let $\cA$ be an operator algebra with the total reduction property and let $\cJ\subset \cA$ be a closed two-sided ideal. Then, $\cJ$ has the total reduction property. Moreover, $\cJ$ admits a bounded approximate identity: there exists a bounded net $(e_i)_i$ in $\cJ$ such that
\[
\lim_i \|ae_i-a\|=\lim_i \|e_i a-a\|=0
\]
for every $a\in \cJ$.
\end{thm}
\begin{proof}
This follows from \cite[Propositions 3.2.7 and 3.3.3]{Gif1997}. 
\end{proof}

Next, we deal with a minor technical detail. Recall that an algebra $\cA\subset B(\cH)$  \emph{acts non-degenerately} if $\ol{\cA \cH}=\cH$.  While operator algebras with the total reduction property do not necessarily act non-degenerately, they  nearly do so.

\begin{lem}\label{lemnondeg}
Let $\cA\subset B(\cH)$ be an operator algebra with the total reduction property. Then, there exists an invertible operator $X\in B(\cH)$ 
with
\[
\|X\| =\|X^{-1}\|\leq 1+\kappa_A(1)
\]
such that 
\[
X\cA X^{-1}=\cA_0\oplus \{0\}
\]
according to some orthogonal decomposition $\cH=\cH_0\oplus \cH_0^\perp$. Moreover, $\cA_0 \subset B(\cH_0)$ is a non-degenerately acting subalgebra with the total reduction property.
\end{lem}
\begin{proof}
Let $M=\ol{\cA\cH}$ which is a closed $\cA$-invariant subspace. There is an idempotent $E\in \cA'$ such that $E \cH= M$ and $\|E\|\leq \kappa_\cA(1).$
Then, it is well-known that there is an invertible operator $X\in B(\cH)$ with 
\[
\|X\|= \|X^{-1}\|\leq 1+\|E\|\leq 1+\kappa_\cA(1)
\]
and such that $X EX^{-1}$ is a self-adjoint projection. The space $X M$ is then reducing for $X \cA X^{-1}$. Put $\cH_0=XM$ and $\cA_0=(X\cA X^{-1})|_{\cH_0}$. We note that
\[
X\cA X^{-1} \cH\subset  X M=\cH_0
\]
so that $X\cA X^{-1}=\cA_0\oplus \{0\}$ according to the decomposition $\cH=\cH_0\oplus \cH_0^\perp$. Thus, $\cA_0$ has the total reduction property by Lemma \ref{lemkapparep}. Moreover, $\cA$ has a bounded approximate identity $(e_i)_i$ by Theorem \ref{thmideal}. Then, if $\xi\in \cH$ and $a\in \cA$ we have that
\[
a\xi=\lim_i e_i a \xi\in \ol{\cA M}
\]
whence $M=\ol{\cA M}$. Thus
\[
\ol{\cA_0 \cH_0}=X \ol{\cA M}=X M=\cH_0
\]
which shows that $\cA_0$ acts non-degenerately on $\cH_0$. 
\end{proof}

We can now extract more information about ideals.

\begin{thm}\label{thmidealw*}
Let $\cA\subset B(\cH)$ be a weak-${}^*$ closed operator algebra with the total reduction property and let $\cJ\subset \cA$ be a weak-${}^*$ closed two-sided ideal.
Then, there is a central idempotent $e\in \cJ\cap \cA'$ such that $\cJ=e\cA$ and $\|e\|\leq \kappa_{\cA}(1)$.
\end{thm}
\begin{proof}
The existence of an idempotent $e\in \cJ\cap \cA'$ such that $\cJ=e\cA$ follows from Lemma \ref{lemnondeg} and \cite[Proposition 3.2.2 and Corollary 3.1.5]{Gif1997}. To get the announced norm estimate, we proceed as follows. Consider the closed subspace $M=e\cH$. Since $e\in \cA'$, we see that $M$ is $\cA$-invariant. Hence, there is another idempotent $f\in \cA'$ with $\|f\|\leq \kappa_{\cA}(1)$ such that $M=f\cH$. Note that $ef=fe$ because $e\in \cJ \subseteq \cA$. Commuting idempotents with identical ranges must be equal, so that indeed $\|e\|\leq \kappa_{\cA}(1)$.
\end{proof}

As a consequence of the previous theorem, we see that if $\cA\subset B(\cH)$ is a weak-${}^*$ closed operator algebra with the total reduction property, then $\cA$ has a unit $u$ with $\|u\|\leq \kappa_{\cA}(1)$. In particular, if $\cJ\subset\cA$ is a weak-${}^*$ closed two-sided ideal, then we have a topological direct sum decomposition
\[
\cA=\cJ+(u-e)\cA
\]
where $\cJ=e\cA$.

We now state the result mentioned in the introduction relating the total reduction property to Kadison's similarity property.
\begin{thm}\label{thmGifKSP}
Let $\fA$ be a $C^*$-algebra.
\begin{enumerate}
\item[\rm{(1)}]  If $\fA$ has Kadison's similarity property, then $\fA$ has the total reduction property.

\item[\rm{(2)}]  If $\fA$ has the total reduction property, then $\fA$ has Kadison's similarity property. More precisely, if $\theta:\fA\to B(\cH)$ is a representation, then there is an invertible operator $X\in B(\cH)$ such that 
\[
a\mapsto X\theta(a)X^{-1}, \quad a\in \fA
\]
is a ${}^*$-homomorphism of $\fA$ and
\[
\|X\| \|X^{-1}\|\leq 128 \kappa_{\fA}(\|\theta\|)^2.
\]
In particular, we see that 
\[
\|\theta\|_{cb}\leq128 \kappa_{\fA}(\|\theta\|)^2.
\]
\end{enumerate}
\end{thm}
\begin{proof}
This is a combination of Lemmas 2.4.1, 2.4.3 and Proposition 2.4.4 in \cite{Gif1997}.
\end{proof}

We close this section with one of the main results of \cite{Gif1997}, which states that if an operator algebra consisting of compact operators has the total reduction property, then it is similar to a $C^*$-algebra. We require a  refined form of a special case of this theorem, which we prove below. 
\begin{thm}\label{thm2.04}
Let $\cH$ be a finite-dimensional Hilbert space and suppose that $\cA\subset B(\cH)$ is an operator algebra with the total reduction property.  Then, there exists an invertible operator $X\in B(\cH)$ with the property that $X\cA X^{-1}$ is a $C^*$-algebra, and such that 
\[
\|X\| \|X^{-1}\|\leq  (1+\kappa_\cA(1))^2  \ 128 (1 + 2 \kappa_{\cA}(1)) \kappa_{\cA}(1 + 2 \kappa_{\cA}(1))^2.
\]
\end{thm}
\begin{proof}
First, we note that by virtue of Lemma \ref{lemnondeg}, we may assume without loss of generality that $\cA$ acts non-degenerately upon conjugating with an invertible operator $V\in B(\cH)$ satisfying
\[
\|V\|=\|V^{-1}\|\leq 1+\kappa_\cA(1).
\]
It is this potential initial conjugation that accounts for the first term of $(1+\kappa_\cA(1))^2$ on the right-hand side of the inequality appearing in the statement.

The proof then consists of a combination of results scattered throughout \cite{Gif1997}. We see that $\cA$ consists of compact operators on $\cH$, and so by \cite[Lemma 4.3.12]{Gif1997} there exist  finitely many minimal idempotents $E_1,\ldots,E_n\in \cA''\cap \cA'$ such that
\[
\cA=E_1\cA E_1+\ldots+E_n \cA E_n
\]
and $\sum_{k=1}^n E_k=I.$ Necessarily we have that these idempotents are pairwise orthogonal  and that for each $k$ the algebra $(E_k \cA E_k)''=E_k \cA''E_k$ contains no proper central idempotent.

By  \cite[Lemmas 1.0.3 and 3.2.3]{Gif1997}, we know that there exists an invertible operator $Y\in B(\cH)$ such that $P_k=YE_k Y^{-1}$ is a self-adjoint projection for every $1\leq k \leq n$, and moreover
\[
\|Y\| \|Y^{-1}\|\leq 1+2\kappa_{\cA}(1).
\]
Note that  
\[
P_k Y\cA Y^{-1}P_k= Y E_k \cA E_k Y^{-1}
\]
for each $1\leq k \leq n$, so we find
\begin{align*}
Y \cA Y^{-1}&=Y E_1\cA E_1Y^{-1}+\ldots+Y E_n \cA E_n Y^{-1}\\
&=\bigoplus_{k=1}^n P_k Y \cA Y^{-1}P_k.
\end{align*}
Moreover, we see that
\[
(P_k Y \cA Y^{-1} P_k)''=Y (E_k \cA E_k)'' Y^{-1}
\]
contains no proper central idempotent for each $1\leq k \leq n$. Using \cite[Lemma 4.3.11]{Gif1997}, for each $k$ we find an invertible operator $Z_k$ such that the algebra $Z_k P_k Y \cA Y^{-1} P_k Z_k^{-1}$ is self-adjoint and
\[
\|Z_k\|= \|Z_k^{-1}\|\leq \sqrt{128} \kappa_{P_k Y \cA Y^{-1} P_k}(1).
\]
Now, by Lemma \ref{lemkapparep} we see that
\[
 \kappa_{P_k Y \cA Y^{-1} P_k}(1)\leq \kappa_\cA(\|Y\| \|Y^{-1}\|)\leq \kappa_{\cA}(1+2\kappa_{\cA}(1))
\]
so that
\[
\|Z_k\|= \|Z_k^{-1}\|\leq \sqrt{128}\kappa_{\cA}(1+2\kappa_{\cA}(1)).
\]
Since the orthogonal projections $P_1,\ldots, P_n$ are pairwise orthogonal and satisfy $\sum_{k=1}^n P_k=I$, if we set $Z=\bigoplus_{k=1}^n P_k Z_kP_k$ then $Z$ is invertible with $Z^{-1}=\bigoplus_{k=1}^n P_k Z^{-1}_kP_k.$ Moreover, we see that
\[
\|Z\| \|Z^{-1}\|\leq 128 \kappa_{\cA}(1+2\kappa_A(1))^2.
\]
Finally, by setting $X=ZY$ we obtain
\[
X\cA X^{-1}=\bigoplus_{k=1}^n P_k Z_k P_k Y\cA Y^{-1} P_k Z^{-1}_k P_k
\]
which is a $C^*$-algebra, and
\[
\|X\| \|X^{-1}\| \leq (1+2\kappa_{\cA}(1))128 \kappa_{\cA}(1+2\kappa_A(1))^2.
\]
\end{proof}

%
%

\section{A reduction to residually finite-dimensional operator algebras}\label{secRed}

This section is meant as motivation for the rest of the paper. The goal here is to show that for amenable operator algebras, the generalized similarity problem can be transplanted to the concrete setting of products of matrix algebras without loss of generality. We will accomplish this by considering cones of operator algebras. 
Recall that if $\cA$ is a Banach algebra, then the \emph{cone of $\cA$} is the Banach algebra
\[
C(\cA) = \{ f: [0,1] \to \cA: f \mbox{ is continuous and } f(0)= 0\}
\]
where if $f\in C(\cA)$ then
\[
\|f\|=\sup_{t\in [0,1]}\|f(t)\|.
\]
Alternatively, we see that $C(\cA)=C_0((0,1])\otimes \cA$, equipped with the injective tensor norm. 
Interestingly, taking the cone of an algebra preserves amenability \cite[Exercise 2.3.6]{rundebook2002}.   We need the following routine fact.

\begin{prop} \label{prop4.09}
Let $\cA$ and $\cB$ be operator algebras and $\theta: \cA \to \cB$ be a representation. Then $\theta$ induces a representation $\Phi_{\theta}: C(\cA) \to C(\cB)$ defined by the formula
\[
(\Phi_{\theta} f)(t) = \theta (f(t)), \quad t \in [0,1].
\]
Furthermore, the map $\theta$ is completely bounded if and only if $\Phi_{\theta}$ is completely bounded, and we have $\|\theta\|_{cb}=\|\Phi_\theta\|_{cb}$.
\end{prop}
\begin{proof}
Let $n\in \bbN$. 
We see that if $f\in \bbM_n(C(\cA))$, then
\[
\|(\Phi^{(n)}_{\theta} f)(t) \|_{\bbM_n(C(\cB))}\leq \|\theta^{(n)}\| \|f(t)\|_{\bbM_n(\cA)}
\]
for every $t\in [0,1]$, so that $\|\Phi^{(n)}_\theta\|\leq \|\theta^{(n)}\|$. 
For the reverse inequality, let $a\in \bbM_n(\cA)$. Define $f_a\in \bbM_n(C(\cA))$ as 
\[
f_a(t)=ta, \quad 0\leq t\leq 1.
\]
Then, we see that $\|f_a\|_{\bbM_n(C(\cA))}=\|a\|_{\bbM_n(\cA)}$ and
\[
\Phi^{(n)}_\theta(f_a)=f_{\theta^{(n)}(a)}.
\]
Hence
\begin{align*}
\|\theta^{(n)}(a)\|_{\bbM_n(\cB)}&=\|f_{\theta^{(n)}(a)}\|_{\bbM_n(C(\cB))}=\|\Phi^{(n)}_\theta(f_a)\|_{\bbM_n(C(\cB))}\\
&\leq \|\Phi^{(n)}_{\theta}\| \|f_a\|_{\bbM_n(C(\cA)}=\|\Phi^{(n)}_{\theta}\|\|a\|_{\bbM_n(\cA)}
\end{align*}
which shows that $ \|\theta^{(n)}\|\leq \|\Phi^{(n)}_\theta\|$. 
\end{proof}

We thus see that to verify whether an operator algebra $\cA$ has the SP property, it is sufficient to check that the cone $C(\cA)$ has it.

Before we proceed further, we introduce some notation which will be used throughout the remainder of the paper. Let $\Lambda \ne \varnothing$ be a set and let $\mathbf{k}: \Lambda \to \bbN$ be a function.  We associate with $\mathbf{k}$ the following $C^*$-algebra:
\[
\cM_{\mathbf{k}} = \prod_{\lambda} \bbM_{k(\lambda)} = \{ (a_\lambda)_\lambda : a_\lambda \in \bbM_{\bf{k}(\lambda)} \mbox{ for all } \lambda \in \Lambda \mbox{ and } \sup_\lambda \norm a_\lambda \norm < \infty \}.\]
Let us also define for each $\lambda\in \Lambda$ the component map
\[
q_\lambda^{\mathbf{k}} : \cM_{\mathbf{k}} \to \bbM_{\bf{k}(\lambda)}
\]
via 
\[
q_\lambda^{\mathbf{k}}(( a_\alpha)_{\alpha}) = a_\lambda.
\]

When $\Lambda$ is a directed set, we shall  use the notation $\cL_{\mathbf{k}}$ instead of $\cM_\mathbf{k}$ in order to emphasize this distinction.
The direction on $\Lambda$ allows us to define the closed, two-sided ideal
\[
\cJ_{\mathbf{k}}=\{(a_\lambda)_\lambda\in \cL_{\mathbf{k}} : \lim_\lambda \|a_\lambda\|=0\}.
\]
It is easily verified that $\cJ_{\bf{k}}$ is nuclear, and hence amenable \cite{Haa1983}. We may now construct the quotient $C^*$-algebra
\[
\cQ_{\mathbf{k}}=\cL_{\mathbf{k}}/\cJ_{\mathbf{k}}.
\]
We let 
\[
\pi_{\mathbf{k}}: \cL_{\mathbf{k}}\to \cQ_{\mathbf{k}}
\]
denote the quotient map. 
The main observation of this section is the following, which is a combination of classical facts. It is a direct adaptation of the discussion found after \cite[Theorem 1]{CFO2014}. We present the proof here for the convenience of the reader.

\begin{thm}\label{thmmatrix}
The following statements are equivalent.
\begin{enumerate}

\item[\rm{(i)}] Every amenable operator algebra has the SP property.

\item[\rm{(ii)}] Let  ${\mathbf{k}}:\Lambda\to \bbN$ and $\mathbf{k}': \Lambda'\to \bbN$ be nets, let $\cD \subset\cL_{\mathbf{k}}$ be an amenable operator algebra, and  let $\theta:\cD\to \cQ_{\mathbf{k'}}$ be a representation.  Then $\theta$ is completely bounded.

\end{enumerate}
\end{thm}
\begin{proof}
We need only prove that (ii) implies (i).  Assume therefore that (ii) holds and that  $\cA\subset B(\cH)$ is an amenable operator algebra.  Let $\theta:\cA\to B(\cH_\theta)$ be a representation. We proceed to show that $\theta$ is completely bounded. 

For this purpose, let  $\fA=C^*(\cA)\subset B(\cH)$. Then, $C(\cA)\subset C(\fA)$. It is easy to see that $C(\fA)$ is homotopic to zero, so that $C(\fA)$ is quasidiagonal by \cite[Theorem 5]{Voi1991} (alternatively, see \cite[Corollary 7.3.7]{BO2008} for the precise statement we need). In particular, by a straightforward adaptation of \cite[Exercise 7.1.3]{BO2008} we may view $C(\fA)$ as a $C^*$-subalgebra of $\cQ_{\mathbf{k}}$ for some net ${\mathbf{k}}: \Lambda\to \bbN$. We conclude that $C(\cA)\subset \cQ_{\mathbf{k}}.$
An identical argument shows that $C(\cB)\subset \cQ_{\mathbf{k}'}$ for some net $\mathbf{k}':\Lambda'\to \bbN$, where $\cB=\ol{\theta(\cA)}$.

By Proposition \ref{prop4.09}, there is a representation 
\[
\Phi_{\theta}: C(\cA) \to C(\cB)
\]
which is completely bounded if and only if $\theta$ is. In turn, it is easily verified that $\Phi_\theta$ is completely bounded if and only if
\[
\Phi_\theta\circ \pi_{\mathbf{k}}: \pi_{\mathbf{k}}^{-1}(C(\cA))\to C(\cB)\subset \cQ_{\mathbf{k}'}
\]
is completely bounded. We know that the cone $C(\cA)$ is amenable.   If we let $\cD = \pi_{\mathbf{k}}^{-1}(C(\cA))$, then we may conclude from \cite[Theorem 2.3.10]{rundebook2002} that $\cD$ is an amenable subalgebra of $\cL_{\mathbf{k}}$. Hence (ii) implies that $\Phi_\theta \circ \pi_{\mathbf{k}}$, and thus $\theta$, is  completely bounded. 
\end{proof}

We remark here that it is plausible that a version of this theorem holds for algebras which merely have the total reduction property. However, a direct adaptation of the proof would require some technology which is unavailable at present, and as such we postpone this interesting issue to future work.

We also emphasize that Theorem \ref{thmmatrix} shows that from the point of view of attempting to solve the generalised similarity problem for amenable operator algebras, it is very meaningful to study amenable subalgebras of products of matrix algebras. We undertake this task for the larger class of operator algebras with the total reduction property, and accordingly we introduce the following convenient terminology.

A subalgebra $\cA \subset B(\cH)$ is said to be \emph{residually finite-dimensional} if there exists a family of finite-dimensional Hilbert spaces $\cH_\lambda$ and a family of completely contractive representations
\[
\varrho_\lambda: \cA \to B(\hilb_\lambda) \]
 such that the map
 \[
a\mapsto \bigoplus_\lambda \rho_\lambda(a), \quad a\in \cA
 \]
 is completely isometric. We mention that this definition is consistent with common usage of the term within the realm of $C^*$-algebras. 
 
It is clear that for any function $\mathbf{k}:\Lambda\to\bbN$, the algebra $\cM_{\mathbf{k}}$ considered above is residually finite-dimensional. Furthermore, subalgebras of residually finite-dimensional operator algebras are residually finite-dimensional as well. We now exhibit a less trivial example, which we will revisit later in the paper.

\begin{eg} \label{egtriangular}
Let $\hilb$ be a Hilbert space and let $(P_\lambda)_{\lambda}$ be a net of finite rank projections increasing strongly to $I$.   Let $\cT\subset B(\cH)$ denote the collection of operators which are triangular with respect to these projections, that is $T \in \cT$ if and only if $P_\lambda T P_\lambda=TP_\lambda$ for every $\lambda$. A standard verification shows that $\cT$ is a weak-${}^*$ closed algebra. For each $\lambda$, the map
 \[
 \varrho_\lambda: \cT \to B(P_\lambda \cH)
 \]
 defined by 
\[
\varrho_\lambda(T) = P_\lambda T P_\lambda, \quad T\in \cT
\]
is a completely contractive homomorphism. Moreover, since $(P_\lambda)_\lambda$ increases to $I$, it is easy to see that $\bigoplus_{\lambda} \varrho_\lambda$ is completely isometric, so that $\cT$ is residually finite-dimensional. 

For future use, we also point that each $\rho_\lambda$ is clearly weak-${}^*$ continuous. By a standard application of the Krein-Smulian theorem \cite[Theorem A.2.5]{BLM2004}, we see that $\bigoplus_{\lambda} \varrho_\lambda$ is a completely isometric weak-${}^*$ homeomorphic algebra homomorphism.
\qed
\end{eg}

 Interestingly, residual finite dimensionality of an operator algebra $\cA\subset B(\cH)$ is not equivalent to that of $C^*(\cA)$, as the following examples show.

\begin{eg} \label{egC^*}
Let $\hilb$ be an infinite-dimensional, separable Hilbert space with orthonormal basis $\{ e_m\}_{m=1}^\infty$.  For each $n\in \bbN$, denote by $P_n$ the orthogonal projection of $\hilb$ onto $\mathrm{span}\{ e_1, e_2, \ldots, e_n\}$.  Let $\cT\subset B(\cH)$ be the algebra of triangular operators with respect to $(P_n)_n$. By Example \ref{egtriangular}, we see that $\cT$ is residually finite-dimensional. On the other hand, it is easy to verify that the ideal of compact operators $\cK(\cH)$ belongs to $C^*(\cT)$. To show that $C^*(\cT)$ is not residually finite-dimensional, it suffices to show that $\cK(\cH)$ is not. But any completely contractive homomorphism of $\cK(\cH)$ is a $*$-homomorphism, and there are no non-zero $*$-homomorphisms from $\cK(\cH)$ into a finite-dimensional $C^*$-algebra, in view of $\cH$ being infinite-dimensional.
\qed
\end{eg}

We close this section by exhibiting a class of residually finite-dimensional $C^*$-algebras which is of particular interest to us, in light of Theorem \ref{thmmatrix}. We suspect that the following statement is well-known to experts, but we provide a proof for the reader's convenience.

\begin{prop}\label{prop4.07}
Let $\mathbf{k}:\Lambda\to \bbN$ be a bounded net. Then, the $C^*$-algebra $\cQ_{\mathbf{k}}$ is residually finite-dimensional. More precisely, there is another set $\Lambda'\neq \varnothing$, a constant function $\mathbf{r}:\Lambda'\to \bbN$ and a completely isometric ${}^*$-homomorphism $\Gamma:\cQ_{\mathbf{k}}\to \cM_{\mathbf{r}}$.
\end{prop}

\begin{proof}
Let $r\in \bbN$ such that $\mathbf{k}(\lambda)\leq r$ for every $\lambda\in \Lambda$. We note that if $m\in \bbN$ and $m\leq r$, then there is a completely isometric $*$-homomorphism
\[
\eps_{m}:\bbM_m\to \bbM_r
\]
defined via
\[
\eps_{m}(a)=a\oplus 0_{r-m}, \quad a\in \bbM_m.
\]
Next, given $b = (b_\lambda)_\lambda + \cJ_{\bf{k}} \in \cQ_{\bf{k}}$, it is easily verified that
\[
\|b\|=\inf_{\mu\in \Lambda} \sup_{\lambda\geq \mu} \|b_\lambda\|.
\]
In particular, we see that $\|b\|$ is a cluster point of $\{\|b_\lambda\|:\lambda\in \Lambda\}$. 
Thus, there exists a cofinal ultrafilter $\cF_b$ on $\Lambda$ for which  $\norm b \norm = \lim_{\lambda \to \cF_b} \norm b_\lambda \norm$.   Closed balls in $\bbM_r$ are compact, so that given $(d_\lambda)_\lambda \in \cL_{\mathbf{k}}$ the limit
\[
\lim_{\lambda \to \cF_b} \eps_{\mathbf{k}(\lambda)}(d_\lambda)
 \]
exists in $\bbM_r$.   Note also that since $\cF_b$ is cofinal, we have that
\[
\lim_{\lambda \to \cF_b} \eps_{\mathbf{k}(\lambda)}(d_\lambda)=0
\]
whenever $(d_\lambda)_\lambda\in \cJ_{\mathbf{k}}$. We may therefore define a map
\[
\gamma_b:\cQ_{\mathbf{k}}\to \bbM_r
\]
such that if $d=(d_\lambda)_\lambda+\cJ_{\mathbf{k}}$ then
\[
\gamma_b(d)=\lim_{\lambda\to \cF_b}\eps_{\mathbf{k}(\lambda)}(d_\lambda).
\]
A routine verification establishes that $\gamma_b$ is a ${}^*$-homomorphism. For any $b\in \cQ_{\mathbf{k}}$ we see that 
\[
\norm \gamma_b(b) \norm = \lim_{\lambda \to \cF_b} \norm \eps_{\mathbf{k}(\lambda)}(b_\lambda) \norm = \lim_{\lambda \to \cF_b} \norm b_\lambda \norm=\norm b \norm \]
by choice of the ultrafilter $\cF_b$. 

Finally, let $\Lambda'$ denote the unit sphere of $\cQ_k$. Define $\mathbf{r}:\Lambda'\to \bbN$ as $\mathbf{r}(b)=r$ for every $b\in \Lambda'$. Then, the map 
\[
\Gamma: \cQ_{\mathbf{k}}\to \cM_{\bf{r}}
\]
defined by
\[
\Gamma(d)=(\gamma_b(d))_{b\in \Lambda'}
\] 
is an isometric $*$-homomorphism, and is thus completely isometric.
\end{proof}

%
%

\section{Residually finite-dimensional operator algebras with the total reduction property} \label{sec02}
Motivated by Theorem \ref{thmmatrix}, in this section we examine in detail the structure of subalgebras of $\cM_{\mathbf{k}}$ with the total reduction property, where $\Lambda \ne \varnothing$ is a set and ${\mathbf{k}}:\Lambda \to \bbN$ is a function. We establish one of our main results based partly on this detailed analysis. 

Our first goal is to show that if $\cA \subset \cM_{\mathbf{k}}$ has the total reduction property, then up to completely bounded isomorphism, we may assume that each component map $q^{\mathbf{k}}_\lambda$ is surjective on $\cA$.

\begin{thm} \label{thm2.05}
Let $\Lambda \ne \varnothing$ be a set and let $\mathbf{k}:\Lambda \to \bbN$ be a function.  Suppose that $\cA\subset \cM_{\mathbf{k}} $ is a subalgebra with the total reduction property.  Then, there exist a set $\Lambda'$, a function $\mathbf{m} :\Lambda' \to \bbN$ and a subalgebra $\cB\subset \cM_\mathbf{m}$ which is completely boundedly isomorphic to $\cA$ and such that for every $\alpha \in \Lambda'$ we have that $q^{{\mathbf{m}}}_\alpha (\cB) = \bbM_{\mathbf{m}(\alpha)}.$ Furthermore, $\cB$ is weak-$*$ closed if $\cA$ is.
\end{thm}

\begin{proof}
For each $\lambda \in \Lambda$, let $\cA_\lambda=q^{{\mathbf{k}}}_\lambda(\cA)$. Note that $\cA_\lambda \subset \bbM_{\mathbf{k}(\lambda)}$ so that $\cA_\lambda$ is necessarily closed and hence has the total reduction property by Lemma \ref{lemkapparep}. Moreover, we see from Lemma \ref{lemkapparep} that
\[
\kappa_{\cA_\lambda}(1)\leq \kappa_{\cA}(\| q^{\mathbf{k}}_\lambda\|)\leq \kappa_{\cA}(1).
\]
By Theorem~\ref{thm2.04}, for each $\lambda \in \Lambda$ there exists an invertible operator $X_\lambda \in \bbM_{\mathbf{k}(\lambda)}$ such that $X_\lambda \cA_\lambda X_\lambda^{-1} \subset\bbM_{\mathbf{k}(\lambda)}$ is a $C^*$-algebra and 
\[
\norm X_\lambda \norm \, \norm X_\lambda^{-1} \norm \le \Delta,
\]
where $\Delta$ is a positive constant depending only on $\kappa_{\cA}(1)$.
Upon rescaling, we may assume that 
\[
\norm X_\lambda \norm  = \|X_\lambda^{-1}\|\leq  \Delta^{1/2}
\]
for each $\lambda \in \Lambda$.  Then, the operator $X = \bigoplus_{\lambda \in \Lambda} X_\lambda \in \cM_{\mathbf{k}}$ is bounded and invertible, and we have
\[
X\cA X^{-1}\subset \bigoplus_{\lambda} X_\lambda \cA_\lambda X_\lambda^{-1}.
\]
For each $\lambda$, there exist a natural number $r_\lambda$ and non-negative integers $d(\lambda,0), d(\lambda, 1), \ldots, d(\lambda,r_\lambda)$  along with an orthogonal decomposition 
\[
\bbC^{\mathbf{k}(\lambda)} =  \bbC^{d(\lambda,0)}\oplus \bbC^{d(\lambda,1)} \oplus \bbC^{d(\lambda,2)} \oplus \cdots \oplus \bbC^{d(\lambda, r_\lambda)}.\]
With respect to this decomposition we must have 
\[
X_\lambda \cA_\lambda X_\lambda^{-1}\subset \{0\}\oplus  \bbM_{d(\lambda,1)} \oplus \bbM_{d(\lambda,2)} \oplus \cdots \oplus \bbM_{d(\lambda,r_\lambda)}
\]
and
\[
q_{d(\lambda,j)}\circ p_\lambda (X_\lambda \cA_\lambda X_\lambda^{-1}) = \bbM_{d(\lambda,j)}
\]
for all $1 \le j \le r_\lambda$, where 
\[
p_\lambda:  \{0\}\oplus  \bbM_{d(\lambda,1)} \oplus \bbM_{d(\lambda,2)} \oplus \cdots \oplus \bbM_{d(\lambda,r_\lambda)}
\to   \bbM_{d(\lambda,1)} \oplus \bbM_{d(\lambda,2)} \oplus \cdots \oplus \bbM_{d(\lambda,r_\lambda)}
\]
and
\[
q_{d(\lambda,j)}: \bbM_{d(\lambda,1)} \oplus \bbM_{d(\lambda,2)} \oplus \cdots \oplus \bbM_{d(\lambda,r_\lambda)}\to \bbM_{d(\lambda,j)}
\]
denote the natural projections. Define $p=\oplus_\lambda p_\lambda$ which is completely isometric $X\cA X^{-1}$ and weak-$*$ homeomorphic on the weak-$*$ closure of $X\cA X^{-1}$. Put $\cB=p(X\cA X^{-1})$. It remains to define the set $\Lambda'$ and the function $\mathbf{m}:\Lambda'  \to \bbN$.  For each $\lambda\in \Lambda$, we define
\[
\Sigma_\lambda=\{(\lambda,j): 1\leq j \leq r_\lambda, d(\lambda,j)\neq 0\}.
\]
We put $\Lambda' = \cup_{\lambda \in \Lambda} \Sigma_\lambda$ and  $m( (\lambda, j)) = d(\lambda, j)$ for all $(\lambda,j)\in \Lambda'$. 
\end{proof}

Next, we make an important observation: simple subalgebras of $\cM_{\bf{k}}$ with the total reduction property are similar to finite-dimensional $C^*$-algebras.

\begin{cor} \label{cor2.07}
Let $\Lambda \ne \varnothing$ be a set, let $\mathbf{k}:\Lambda \to \bbN$ be a function and let $\cA \subset \cM_{\mathbf{k}}$ be a simple subalgebra which has the total reduction property.   Then, $\cA$ is similar to a finite-dimensional $C^*$-algebra.
\end{cor}

\begin{proof}
Invoking \cite[Theorem 1.10]{Haa1983s}, we see that it is sufficient to prove the statement for a completely boundedly isomorphic image of $\cA$. Hence, by virtue of Theorem~\ref{thm2.05} we may assume that  $q^{\mathbf{k}}_\lambda(\cA) = \bbM_{\mathbf{k}(\lambda)}$ for every $\lambda \in \Lambda$. Fix $\lambda_0 \in \Lambda$, and consider the surjective, contractive homomorphism $q_{\lambda_0}^{\mathbf{k}}|_\cA:\cA\to \bbM_{\mathbf{k}(\lambda_0)}$. Since $\cA$ is simple, we see that $q^{\mathbf{k}}_{\lambda_0}|_\cA$ is invertible, and hence that $\cA$ is boundedly isomorphic to the finite-dimensional $C^*$-algebra $\bbM_{\mathbf{k}(\lambda_0)}$. We conclude that $\cA$ must be similar to a ${}^*$-isomorphic image of $\bbM_{\mathbf{k}(\lambda_0)}$~\cite{Wri1980}.
\end{proof}

We mention in passing that Corollary \ref{cor2.07} can be extended to cover the case where the algebra possesses finitely many ideals. We leave the details to the interested reader.

Before proceeding with the main result of this section, we require two preliminary facts. We first establish a useful estimate.

\begin{lem} \label{lem2.10}
Let $\Lambda \ne \varnothing$ be a set and let $\mathbf{k}: \Lambda \to \bbN$ be the constant function 
\[
\mathbf{k}(\lambda) = k, \quad \lambda \in \Lambda
\]
for some fixed $k\in \bbN$. Let $(X_\lambda)_{\lambda \in \Lambda}$ be  a collection of invertible operators in  $\bbM_{k}$ and let
\[
\cA = \left\{ \bigoplus_{\lambda \in \Lambda} X_\lambda^{-1} T X_\lambda : T \in \bbM_{k}\right\}. \]
Suppose that $\cA \subset \cM_{\mathbf{k}}$, that $\cA$ has the total reduction property and that there exists $\lambda_0 \in \Lambda$ for which $X_{\lambda_0} = I$.  Then 
\[
\sup_\lambda \norm X_\lambda \norm \, \norm X_\lambda^{-1} \norm \le 4 \kappa_\cA(1)^2. \]
\end{lem}

\begin{proof}
Upon rescaling if necessary, we may assume that 
\[
\norm X_\lambda \norm = \norm X_\lambda^{-1} \norm=\norm X_\lambda \norm^{1/2} \, \norm X_\lambda^{-1}\| ^{1/2}
\]
for all $\lambda \in \Lambda$.   For each  $\nu \in \Lambda$ with $\nu\neq \lambda_0$, let
\[
\Gamma_\nu: \cA \to \bbM_{2k}
\]
be defined by
\[
\Gamma_\nu \left(\bigoplus_{\lambda \in \Lambda} X_\lambda^{-1} T X_\lambda \right) = T \oplus X_\nu^{-1} T X_\nu \]
for every $T\in \bbM_k$. Observe that since $X_{\lambda_0}=I$, each such map $\Gamma_\nu$ is a contractive homomorphism. Therefore $\Gamma_\nu (\cA)$
has the total reduction property by Lemma \ref{lemkapparep}.
Now consider the subspace 
\[
W_\nu = \{ \xi \oplus X_\nu^{-1} \xi: \xi \in\bbC^k\}.
\]
Clearly, $W_\nu$ is invariant for $\Gamma_\nu(\cA)$, so there exists an idempotent $E_\nu \in \Gamma_\nu (\cA)^\prime $ with $\norm E_\nu \norm\le \kappa_{\cA}(1)$ and whose range is $W_\nu$. Routine calculations show that
\[
 \Gamma_\nu (\cA)^\prime  = \left\{ \begin{bmatrix} \alpha I_{k} & \beta X_\nu \\ \gamma X_\nu^{-1} & \delta I_k \end{bmatrix}: \alpha, \beta, \gamma, \delta \in \bbC \right\}\subset \bbM_{2k}. 
\]
Using the fact that the range of $E_\nu$ is $W_\nu$, we infer that 
\[
E_\nu = \begin{bmatrix} \alpha_\nu I_k & \beta_\nu X_\nu \\ \alpha_\nu X_\nu^{-1} & \beta_\nu I_k \end{bmatrix}\]
for an appropriate choice of $\alpha_\nu, \beta_\nu \in \bbC$. Moreover, since $E_\nu$ is idempotent and  $\dim\, W_\nu = k$, we deduce that $k=\text{tr}(E_\nu)$. On the other hand, we have that $\text{tr}(E_\nu)=(\alpha_\nu+\beta_\nu)k$, so that $\alpha_\nu + \beta_\nu = 1$. In particular 
\[
\max\{|\alpha_\nu|, |\beta_\nu| \}\ge 1/2
\]
whence
\[
\norm E_\nu \norm \ge \max \{\norm \alpha_\nu X_\nu^{-1} \norm, \norm \beta_\nu X_\nu \norm \} \ge  \norm X_\nu\norm /2.
\] 
Thus $\norm X_\nu \norm \le 2 \norm E_\nu \norm \le 2\kappa_{\cA}(1)$.   We conclude that
\[
\norm X_\nu \norm \ \norm X_\nu^{-1} \norm =\|X_\nu\|^2 \le 4 \kappa_{\cA}(1)^2
\]
for every $\nu \in \Lambda \setminus \{ \lambda_0\}$.  Since $X_{\lambda_0} = I_k$ and $\kappa_{\cA}(1) \ge 1$, we are done.
\end{proof}

We also need a property of central idempotents in weak-${}^*$ closed algebras having the total reduction property.

\begin{lem}\label{lemsum}
Let $\cA\subset B(\cH)$ be a weak-${}^*$ closed algebra with the total reduction property and with unit $u\in \cA$. Let $(e_i)_{i\in I}$ be a family of central idempotents in $\cA$ such that
\[
\cap_{i\in I} (u-e_i)\cA=\{0\}.
\]
Let $\cS\subset B(\cH)$ denote the smallest weak-${}^*$ closed subspace containing $e_i \cA$ for every $i\in I$. Then, $\cA=\cS$.
\end{lem}
\begin{proof}
We note that $e_i \cA\subset \cA$ for every $i\in I$, so that $\cS\subset \cA$ since $\cA$ is weak-${}^*$ closed. It is readily checked that $\cS$ is a weak-${}^*$ closed two-sided ideal, and thus by Theorem \ref{thmidealw*} there is a central idempotent $f\in \cA$ with $f\cA=\cS$.  Note that $e_i=e_i u\in e_i \cA\subset \cS$, so that $e_i f=e_i$ and $e_i(u-f)=0$ for every $i\in I$ . We claim that $f\cA=\cA$. Indeed, if $a\in \cA$ then we see that 
\[
(u-f)a=(u-e_i)(u-f)a
\]
for every $i\in I$, so that 
\[
(u-f)a\in \cap_{i\in I} (u-e_i)\cA=\{0\}
\]
whence $fa=a$. We conclude that $\cA=f\cA=\cS$.
\end{proof}

We can now prove one of the main results of the paper.

\begin{thm} \label{thm2.13}
Let $\Lambda \ne \varnothing$ be a set and let $\mathbf{k}:\Lambda\to \bbN$ be a function.  Suppose that $\cA \subset \cM_{\mathbf{k}}$ is a subalgebra which has the total reduction property.  
If $\cA$ is weak-${}^*$ closed, then $\cA$ is similar to a $C^*$-algebra.
\end{thm}
\begin{proof}
We start by noting once again that it is sufficient to prove the statement for a completely boundedly isomorphic image of $\cA$. Thus, by virtue of Theorem~\ref{thm2.05} we may assume that  $\cA$ is a weak-$*$ closed subalgebra of $\cM_k$ such that $q^{\mathbf{k}}_\lambda(\cA) = \bbM_{\mathbf{k}(\lambda)}$ for every $\lambda \in \Lambda$.

For each $\lambda\in \Lambda$, consider the weak-${}^*$ closed ideal $\cJ_\lambda = \ker q^{\mathbf{k}}_\lambda$.  Then 
\[
\cA / \cJ_\lambda \simeq q^{\mathbf{k}}_\lambda(\cA)=  \bbM_{\mathbf{k}(\lambda)}.
 \]  
Moreover, by Theorem \ref{thmidealw*} there exists a central idempotent $f_\lambda \in\cJ_\lambda \cap \cA'$ with $\cJ_\lambda=f_\lambda \cA$, $\|f_\lambda\|\leq \kappa_{\cA}(1)$ and such that we have a topological direct sum decomposition
\[
\cA=\cJ_\lambda+(u-f_\lambda)\cA
\]
where $u$ denotes the unit of $\cA$ and satisfies $\|u\|\leq \kappa_{\cA}(1)$ (see the discussion following Theorem \ref{thmidealw*}).
Set now $e_\lambda = (u - f_\lambda)$ for each $\lambda \in \Lambda$, and  put $\cK_\lambda = e_\lambda \cA$ so that 
\[
\cA = \cJ_\lambda+ \cK_\lambda.
\]
By Lemma \ref{lemkapparep} we see that $\cK_\lambda$ has the total reduction property and that 
\[
\kappa_{\cK_\lambda}(1)\leq \kappa_{\cA}(\|e_\lambda\|)\leq \kappa_{\cA}(2\kappa_\cA(1)).
\]
Moreover, we note that
\[
\cK_\lambda \simeq \cA/\cJ_\lambda \simeq \bbM_{\mathbf{k}(\lambda)}
\]
and consequently $\cK_\lambda$ is simple and a minimal ideal of $\cA$. Hence, for each $\alpha\in \Lambda$ we must have that $q^{\mathbf{k}}_\alpha(\cK_\lambda)$ is a two-sided ideal of $\bbM_{\mathbf{k}(\alpha)}$, and therefore is either $\{0\}$ or 
$\bbM_{\mathbf{k}(\alpha)}$. Equivalently, $q^{\mathbf{k}}_\alpha(e_\lambda) = 0$ or $q^{\mathbf{k}}_\alpha(e_\lambda) = I$.
  
For each $\lambda \in \Lambda$ let
\[
\Delta_\lambda = \{ \alpha \in \Lambda: q^{\mathbf{k}}_\alpha(e_\lambda) = I \}.
\]
Note that  for every $\lambda \in \Lambda$ we have $q^{\mathbf{k}}_\lambda(u)=I$ since $q^{\mathbf{k}}_\lambda(\cA) = \bbM_{\mathbf{k}(\lambda)}$. Recall that 
\[
f_\lambda \cA=\cJ_\lambda = \ker q^{\mathbf{k}}_\lambda
\]
 thus
\[
0=q^{\mathbf{k}}_\lambda(f_\lambda)=I-q^{\mathbf{k}}_\lambda(e_\lambda)
\]
so that $\lambda \in \Delta_\lambda$ for every $\lambda \in \Lambda$. In particular, we see that
\[
\Lambda=\cup_{\lambda \in \Lambda}\Delta_\lambda.
\]
If $\lambda_1, \lambda_2 $ are distinct elements of $\Lambda$,  then by minimality we must have either $\cK_{\lambda_1} = \cK_{\lambda_2}$ or $\cK_{\lambda_1} \cap \cK_{\lambda_2} = \{ 0 \}$.
Note also that $\cK_{\lambda_1} \cap \cK_{\lambda_2} = \{ 0 \}$ if and only if $e_{\lambda_1}e_{\lambda_2}=0$, and $\cK_{\lambda_1}= \cK_{\lambda_2}$ if and only if $e_{\lambda_1}=e_{\lambda_2}$. We conclude that if $\lambda_1,\lambda_2 \in \Lambda$, then either $\Delta_{\lambda_1} = \Delta_{\lambda_2}$, or $\Delta_{\lambda_1} \cap \Delta_{\lambda_2} = \varnothing$. The equivalence relation
\[
\lambda_1\sim \lambda_2 \quad \text{if and only if} \quad e_{\lambda_1}=e_{\lambda_2}
\]
partitions $\Lambda$ into a set $\Omega$ of disjoint subsets. For each $\omega\in \Omega$, choose $\lambda_\omega\in \omega$. Then, we have the following disjoint union
\[
\cup_{\omega\in \Omega}\Delta_{\lambda_\omega}=\Lambda.
\]
If $\omega_1,\omega_2\in \Omega$ are distinct, then $e_{\lambda_{\omega_1}}e_{\lambda_{\omega_2}}=0$. Let now $a\in \cA$ and assume that $e_{\lambda_\omega}a=0$ for every $\omega\in \Omega$. Then, we see that $q^{\mathbf{k}}_\alpha(a)=0$ for every $\alpha \in \cup_{\omega\in \Omega}\Delta_{\lambda_\omega}=\Lambda$, so that $a=0$. By Lemma \ref{lemsum}, we conclude that $\cA$ is the smallest weak-${}^*$ closed subspace containing $e_{\lambda_\omega}\cA$ for every $\omega\in \Omega$. On the other hand,  if $\omega_1,\omega_2\in \Omega$ are distinct, then $\Delta_{\lambda_{\omega_1}}\cap \Delta_{\lambda_{\omega_2}}=\varnothing$. Hence, we find that
\[
\cA=\bigoplus_{\omega\in \Omega} e_{\lambda_\omega}\cA=\bigoplus_{\omega\in \Omega} \cK_{\lambda_\omega}.
\]

Now, for each $\lambda\in \Lambda$ we have that $q^{\mathbf{k}}_\alpha(\cK_\lambda) = \bbM_{\mathbf{k}(\alpha)}$ for all $\alpha \in \Delta_\lambda$.  Since $\cK_\lambda$ is simple, we see that $q^{\mathbf{k}}_\alpha|_{\cK_\lambda}$ is a bounded isomorphism between $\cK_\lambda$ and $ \bbM_{\mathbf{k}(\alpha)}$ for every $\alpha \in \Delta_\lambda$. In particular, $\mathbf{k}$ is constant on $\Delta_\lambda$. In addition, if we let $\mu_\lambda$ be a fixed element of $\Delta_\lambda$, then we see that for each $\alpha \in \Delta_\lambda$ there is an invertible operator $X_{\lambda,\alpha}\in \bbM_{\mathbf{k}(\alpha)}$ such that
\[
X_{\lambda, \alpha}q^{\mathbf{k}}_{\mu_\lambda}(a)X_{\lambda, \alpha}^{-1}=q^{\mathbf{k}}_{\alpha}(a)
\]
for every $a\in \cK_\lambda$. Since
\[
\cK_\lambda=\left\{\bigoplus_{\alpha\in \Lambda}q^{\mathbf{k}}_{\alpha}(a):a\in \cK_\lambda \right\},
\]
we see that up to a (unitary) reordering of the components 
\[
\cK_\lambda =\bigoplus_{\alpha \in \Lambda \setminus\Delta_\lambda} \{0\}\oplus \left\{ \bigoplus_{\alpha \in \Delta_\lambda }X_{\lambda, \alpha}bX_{\lambda, \alpha}^{-1}:b\in  \bbM_{\mathbf{k}(\mu_\lambda)} \right\}.
\]
By Lemmas \ref{lemkapparep} and \ref{lem2.10}, we find that 
\[
\norm X_{\lambda, \alpha}^{-1} \norm \ \norm X_{\lambda, \alpha} \norm \le 4 \kappa_{\cK_\lambda}(1)^2  \le 4 \kappa_{\cA}(2\kappa_\cA(1))^2 \]
for every $\alpha \in \Delta_\lambda$.   Upon  rescaling we may assume that
\[
\norm X_{\lambda, \alpha} \norm = \norm X_{\lambda, \alpha}^{-1} \norm\leq 2\kappa_{\cA}(2\kappa_\cA(1)).
\]
for every $\alpha\in \Delta_\lambda$. 

Hence, the operator
\[
X=\bigoplus_{\omega \in \Omega}\left(\bigoplus_{\alpha\in \Lambda \setminus \Delta_{\lambda_\omega}}I\oplus \bigoplus_{\alpha \in \Delta_{\lambda_\omega}} X_{\lambda_\omega,\alpha}\right)
\]
is invertible with bounded inverse. Finally, we find
\begin{align*}
X^{-1}\cA X=&X^{-1}\left( \bigoplus_{\omega\in \Omega}\cK_{\lambda_\omega}\right)X\\
&=\bigoplus_{\omega\in \Omega}\left(\bigoplus_{\alpha \in \Lambda \setminus\Delta_{\lambda_\omega}} \{0\}\oplus \left\{ \bigoplus_{\alpha \in \Delta_{\lambda_\omega}}b:b\in  \bbM_{\mathbf{k}(\mu_{\lambda_\omega})} \right\} \right)\\
&\simeq  \bigoplus_{\omega\in \Omega} \bbM_{\mathbf{k}(\mu_{\lambda_\omega})}
\end{align*}
which is a $C^*$-algebra.
\end{proof}

The reader will notice that the reason we require the algebra above to be weak-${}^*$ closed is to avail ourselves of Theorem \ref{thmidealw*}. We now record a straightforward consequence of the previous result.

\begin{cor} \label{cor2.14}
Let $\Lambda \ne \varnothing$ be a set and let $\mathbf{k}: \Lambda \to \bbN$ be a function.  Suppose that $\cA \subset \cM_{\mathbf{k}}$ is a subalgebra which has the total reduction property.  If $\cA$ is weak-$*$ closed, then it has the SP property .
\end{cor}
\begin{proof}
We know that $\cA$ is similar to a $C^*$-algebra by Theorem \ref{thm2.13}. Next, recall that $C^*$-algebras with the total reduction property have Kadison's similarity property by Theorem \ref{thmGifKSP}, so that $\cA$ must have the SP property.
\end{proof}

Finally, we close this section with an application to triangular algebras (see Example \ref{egtriangular} for the definition).

\begin{cor}\label{cortriang1}
Let $\cA\subset B(\cH)$ be a weak-${}^*$ closed triangular operator algebra with the total reduction property. Then, $\cA$ is similar to a $C^*$-algebra.
\end{cor}
\begin{proof}
By Example \ref{egtriangular}, we know that there is a completely isometric weak-${}^*$ homeomorphic homomorphism $\Phi: \cA\to \cL_{\bf{}k}$, for some net $\mathbf{k}:\Lambda\to \bbN$. Thus, by Lemma \ref{lemkapparep} and Theorem \ref{thm2.13} we see there is an invertible operator $X$ such that $X\Phi(\cA)X^{-1}$ is a $C^*$-algebra. The map
\[
X\Phi(a)X^{-1}\mapsto a, \quad a\in \cA
\]
is completely bounded on a $C^*$-algebra, and thus is similar to a ${}^*$-homomorphism \cite{Haa1983s}. Consequently, $\cA$ is similar to the image of a $C^*$-algebra under a ${}^*$-homomorphism, and so is similar to a $C^*$-algebra.
\end{proof}
	

\section{Representations with residually finite-dimensional range} \label{sec04}

In the previous section, we proved that a weak-$*$ closed residually finite-dimensional operator algebra with the total reduction property has the SP property: all its representations are automatically completely bounded. In other words, we restricted our attention to the case where  the domains of the representations are contained in a product of matrix algebras. In this section, we shift our focus to the range and assume that it is residually finite-dimensional, while the domain is allowed to be an arbitrary operator algebra with the total reduction property.

The driving force behind our efforts in this section is the following simple observation, which we record for ease of reference.

\begin{lem} \label{lem4.05}
Let  $\cA$ and $\cB$ be operator algebras. Assume that $\cB$ is residually finite-dimensional and let $(\rho_\lambda)_\lambda$ be the corresponding family of completely contractive representations of $\cB$ such that $\bigoplus_\lambda \varrho_\lambda$ is completely isometric.  Let $\theta: \cA \to \cB$ be a representation. Then $\theta$ is completely bounded if and only if $\sup_\lambda \norm \varrho_\lambda \circ \theta \norm_{cb} < \infty$.
\end{lem}

Although completely elementary, this lemma shows that to determine whether a representation with residually finite-dimensional range is completely bounded, we may restrict our attention to finite-dimensional representations. In particular, we have the following consequence which the reader may want to compare with Theorem \ref{thmmatrix}.

\begin{thm} \label{thm4.08}
Let $\cA$ be an operator algebra and let $\mathbf{k}:\Lambda\to \bbN$ be a bounded net. If $\theta: \cA\to \cQ_\mathbf{k}$ is a bounded representation, then $\theta$ is completely bounded.
\end{thm}

\begin{proof}
By Proposition~\ref{prop4.07} we can find another set $\Lambda'\neq \varnothing$ and a constant function $\mathbf{r}:\Lambda'\to \bbN$ for which there exists a completely isometric ${}^*$-homomorphism 
\[
\Gamma: \cQ_{\mathbf{k}}\to \cM_{\bf{r}}.
\]
It suffices to show that $\Gamma\circ \theta$ is completely bounded. To see this, use Lemma \ref{lem4.05} to conclude that $\theta$ is completely bounded if and only if
\[
\sup_\lambda \|q^{\mathbf{k}}_\lambda\circ\theta\|_{cb}<\infty.
\]
Now, a classical result of R.R.~Smith~\cite{Smi1983} shows that  for any $\lambda\in \Lambda$ we have
	\begin{align*}
	\norm q^{\mathbf{k}}_\lambda \circ \theta \norm_{cb} 
		\le \mathbf{k}(\lambda) \norm q^{\mathbf{k}}_\lambda \circ \theta \norm \leq  \left(\sup_{\lambda\in \Lambda}\mathbf{k}(\lambda)\right) \norm \theta \norm.
	\end{align*}
The proof is complete.
\end{proof}


\bigskip

The remainder of this section is devoted to establishing another one of our main results dealing with representations of operator algebras with the total reduction property whose range are residually finite-dimensional. The key technical tool is the following observation, which generalizes the fact that finite-dimensional $C^*$-algebras can be faithfully represented on finite-dimensional Hilbert spaces. We suspect it is well-known. The idea is reminiscent of that found in the proof of \cite[Theorem 6.3]{Pis2003}, which apparently has its origins in the work of Dixon. We also note that the classical Blecher-Ruan-Sinclair theorem \cite{BRS1990} may not produce a finite-dimensional Hilbert space and thus does not meet our specific needs.

\begin{prop}\label{prop4.11}
Let $\cA$ be a finite-dimensional operator algebra, and fix $\varepsilon>0$ and $d\in \bbN$. Then, there exist a finite-dimensional Hilbert space $\cH$ and a completely contractive homomorphism $\phi:\cA\to B(\cH)$ with 
\[
\|\phi^{(n)}(A)\|_{\bbM_n(B(\cH))}\geq (1-\varepsilon) \|A\|_{\bbM_n(\cA)}
\] 
for every $A\in \bbM_n(\cA)$ and every $1\leq n \leq d$.
\end{prop}

\begin{proof}
By considering the unitization of $\cA$ (which is also finite-dimensional) if necessary, we may assume that $\cA$ is unital \cite{BLM2004}. Since $\cA$ is finite-dimensional, its closed unit ball is compact. Hence, we may choose $\alpha_1,\ldots,\alpha_N\in \cA$ such that $\|\alpha_k\|\leq 1$ for every $1\leq k\leq N$, and with the property that for every $a\in \cA$ with $\|a\|\leq 1$ there is $1\leq k \leq N$ such that
\[
\|a-\alpha_k\|<\varepsilon/d.
\]
In particular, given $A\in \bbM_n(\cA)$ with $\|A\|_{\bbM_n(\cA)}=1$ where $1\leq n \leq d$, we can find $A'=(a'_{ij})_{ij}\in \bbM_n(\cA)$ with the property that each $a'_{ij}$ belongs to the set $\{\alpha_1,\ldots,\alpha_N\}$ and
\[
\|A-A'\|_{\bbM_n(\cA)}<\varepsilon.
\]
Next, note that there is a finite-dimensional subspace $\cH_1$ with the property that
\[
\|P^{(n)}_{\cH_1}A^* AP^{(n)}_{\cH_1}\|\geq (1-\varepsilon)\|A\|^2_{\bbM_n(\cA)}
\]
whenever $A=(a_{ij})_{ij}\in \bbM_n(\cA)$ for some $1\leq n \leq d$ and each $a_{ij}$ belongs to the set $\{\alpha_1,\ldots,\alpha_N\}$.
Define a unital completely positive map
\[
\omega:C^*(\cA)\to B(\cH_1)
\]
as $\omega(t)=P_{\cH_1}t|_{\cH_1}$ for every $t\in C^*(\cA)$. Next, we carefully analyze the Stinespring dilation of $\omega$, and for that purpose we briefly recall the details of its construction.

Define a positive semi-definite bilinear form on the vector space $C^*(\cA)\otimes \cH_1$ as
\[
\Psi \left(\sum_{i} t_i \otimes h_i, \sum_{j} s_j\otimes k_j\right)= \sum_{i,j}\langle \omega(s_j^* t_i)h_i,k_j\rangle_{\cH_1}.
\]
Let 
\[
\cZ=\{ v\in C^*(\cA)\otimes \cH_1: \Psi(v,v)=0\},
\]
which is a subspace of $C^*(\cA)\otimes \cH_1$. The quotient $(C^*(\cA)\otimes \cH_1)/\cZ$ is an inner product space, and we denote its completion by $\cE$. Given $v\in C^*(\cA)\otimes \cH_1$ we denote its image in $\cE$ by $[v]$. Define
\[
\pi:C^*(\cA)\to B(\cE)
\]
via 
\[
\pi(t)\left[\sum_j s_j\otimes h_j\right]=\sum_j ts_j\otimes h_j
\]
for every $t\in C^*(\cA)$. The complete positivity of $\omega$ implies that $\pi$ is well-defined. Furthermore, it is readily verified that $\pi$ is a unital $*$-homomorphism. Now, here is the key point: we denote by $\cH$ the finite-dimensional space spanned by the elements of the form $[a\otimes h]$ for $a\in \cA$ and $h\in \cH_1$. Since $\cA$ is an algebra, the space $\cH$ is invariant for $\pi(\cA)$. Hence, we may define a unital completely contractive homomorphism
\[
\phi:\cA\to B(\cH)
\]
as $\phi(a)=\pi(a)|_{\cH}$ for every $a\in \cA$. It only remains to establish the announced lower bound.

Notice that for $h,k\in \cH_1$ we have
\[
\Psi(1\otimes h,1\otimes k)=\langle \omega(1)h,k\rangle_{\cH_1}=\langle h,k\rangle_{\cH_1}
\]
and thus $\|[1\otimes h]\|_{\cH}=\|h\|_{\cH_1}$ for every $h\in \cH_1$. In particular, if $\xi=(\xi_1,\ldots,\xi_n)^\mathrm{t} \in \cH_1^{(n)}$ satisfies $\|\xi\|_{\cH_1^{(n)}}=1$ then the vector
\[
\Xi=(1\otimes \xi_1,\ldots,1\otimes \xi_n)^{\mathrm{t}} \in \cH^{(n)}
\]
also satisfies $\|\Xi\|_{\cH^{(n)}}=1$. Furthermore, if $A\in \bbM_n(\cA)$ then a routine calculation shows that
\begin{align*}
 \langle \omega^{(n)}(A^* A)\xi,\xi\rangle_{\cH_1^{(n)}}= \| \pi^{(n)}(A) \Xi\|_{\cE^{(n)}} ^2=\|\phi^{(n)}(A)\Xi\|_{\cH^{(n)}}^2.
\end{align*}
Hence,
\[
\langle \omega^{(n)}(A^* A)\xi,\xi\rangle_{\cH_1^{(n)}}\leq \|\phi^{(n)}(A)\|^2_{\bbM_n(B(\cH))}
\]
and we conclude that
\[
\|\phi^{(n)}(A)\|_{\bbM_n(B(\cH))}^2\geq \|\omega^{(n)}(A^*A)\|_{\bbM_n(B(\cH_1))}=\|P^{(n)}_{\cH_1}A^* AP^{(n)}_{\cH_1}\|\geq (1-\varepsilon) \|A\|^2_{\bbM_n(\cA)}
\]
whenever $A=(a_{ij})_{ij}\in \bbM_n(\cA)$
where $1\leq n \leq d$ and each $a_{ij}$ belongs to the set $\{\alpha_1,\ldots,\alpha_N\}$.
As noted above, given a general element $A\in \bbM_n(\cA)$ with $1\leq n \leq d$ and $\|A\|_{\bbM_n(\cA)}=1$, we can find $A'=(a'_{ij})_{ij}\in \bbM_n(\cA)$ where each $a'_{ij}$ belongs to the set $\{\alpha_1,\ldots,\alpha_N\}$ and
\[
\|A-A'\|_{\bbM_n(\cA)}< \varepsilon.
\]
In particular, $\|A'\|_{\bbM_n(\cA)}\geq 1-\varepsilon$.
Since $\phi$ is completely contractive, we find
\[
\|\phi^{(n)}(A)\|_{\bbM_n(B(\cH))}\geq \|\phi^{(n)}(A')\|_{\bbM_n(B(\cH))}-\varepsilon\geq (1-\varepsilon)^{3/2}-\varepsilon
\]
and the proof is complete.
\end{proof}

This theorem allows us to prove a uniform estimate for certain representations of operator algebras with the total reduction property. 

\begin{thm}\label{thm4.12}
Let $\cA$ be an operator algebra with the total reduction property and let $\theta:\cA\to B(\cH)$ be a representation. Assume that $\cA/\ker \theta$ is finite-dimensional. Then, there is a positive constant $\Delta$ depending only on $\kappa_\cA$ and $\|\theta\|$ such that $\|\theta\|_{cb}\leq \Delta.$
\end{thm}

\begin{proof}
Let $\widehat{\theta}:\cA/\ker \theta\to B(\cH)$ be the representation induced by $\theta$. Then, we have $\|\theta\|_{cb}=\|\widehat{\theta}\|_{cb}$.  By Proposition~\ref{prop4.11}, there is a finite-dimensional Hilbert space $\cH'$, an operator algebra $\cB\subset B(\cH')$ and completely contractive isomorphism
\[
\phi:\cA/\ker \theta\to\cB 
\]
with $\|\phi^{-1}\|\leq 2$. Therefore, $\cB$ has the total reduction property and 
\[
\kappa_\cB(t)\leq \kappa_{\cA/\ker \theta}(\|\phi\| t)\leq \kappa_{\cA}(t)
\]
for every $t\geq 0$, by Lemma \ref{lemkapparep}. Then, an application of Theorem~\ref{thm2.04} yields an invertible operator $X\in B(\cH')$ such that 
$X\cB X^{-1}$ is a $C^*$-algebra and
\[
\|X\| \|X^{-1}\|\leq \Delta_0
\]
for some positive constant $\Delta_0$ depending only on $\kappa_\cA(1)$. We will use the following notation
\[
\mathrm{Ad}_X(T)=XTX^{-1}, \quad T\in B(\cH').
\]
We see that
\[
\|\mathrm{Ad}_X\|_{cb}\leq \Delta_0, \quad \|\mathrm{Ad}^{-1}_X\|_{cb}\leq \Delta_0.
\]
Then, $X\cB X^{-1}=\mathrm{Ad}_X(\cB)$ has the total reduction property and
\begin{align*}
\kappa_{X\cB X^{-1}}(t)&\leq \kappa_{\cB}(\|\mathrm{Ad}_X\| t)\\
&\leq \kappa_{\cA}(\Delta_0 t)
\end{align*}
for every $t\geq 0$, again by Lemma \ref{lemkapparep}.

Now, by virtue of Theorem \ref{thmGifKSP}, we see that the map
\[
\widehat{\theta}\circ \phi^{-1}\circ \mathrm{Ad}_X^{-1}:X\cB X^{-1}\to B(\cH)
\]
is completely bounded with
\begin{align*}
\|\widehat{\theta}\circ \phi^{-1}\circ \mathrm{Ad}_X^{-1}\|_{cb}&\leq  128\kappa_{X\cB X^{-1}}(\|\widehat{\theta}\circ \phi^{-1}\circ \mathrm{Ad}_X^{-1}\|)^2\\
&\leq  128\kappa_{\cA}(\Delta_0 \|\widehat{\theta}\circ \phi^{-1}\circ \mathrm{Ad}_X^{-1}\|)^2\\
&\leq 128\kappa_{\cA}(\Delta_0 \|\phi^{-1}\| \|\mathrm{Ad}_X^{-1} \| \|\theta\|)^2\\
&\leq  128\kappa_{\cA}(2\Delta^2_0 \|\theta\|)^2.
\end{align*}
Finally, note that
\[
\widehat{\theta}=\widehat{\theta}\circ \phi^{-1}\circ \mathrm{Ad}_X^{-1}\circ \mathrm{Ad}_X\circ \phi.
\]
so that
\begin{align*}
\|\widehat{\theta}\|_{cb}&\leq \| \mathrm{Ad}_X\|_{cb} \|\phi\|_{cb}\| \widehat{\theta}\circ \phi^{-1}\circ \mathrm{Ad}_X^{-1}\|_{cb}\\
&\leq  128 \Delta_0\kappa_{\cA}(2\Delta^2_0 \|\theta\|)^2
\end{align*}
so we may take 
\[
\Delta=128 \Delta_0\kappa_{\cA}(2\Delta^2_0 \|\theta\|)^2.
\]
\end{proof}

We can now establish the main result of this section, which says that for operator algebras with the total reduction property, representations with residually finite-dimensional ranges are necessarily completely bounded.

\begin{cor}\label{cor4.15}
Let $\Lambda\neq \varnothing$ be a set and let $\mathbf{k}:\Lambda\to \bbN$ be a function. Let $\cA$ be an operator algebra with the total reduction property and let $\theta:\cA\to \cM_{\bf{k}}$ be a representation. Then, $\theta$ is completely bounded.
\end{cor}
\begin{proof}
Combine Lemma \ref{lem4.05} and Theorem \ref{thm4.12}.
\end{proof}

We give an application to triangular algebras (see Example \ref{egtriangular} for the definition).

\begin{cor}\label{cortriang2}
Let $\cA$ be an operator algebra with the total reduction property and let $\theta:\cA\to B(\cH)$ be a representation such that $\theta(\cA)$ is a triangular algebra. Then, $\theta$ is completely bounded.
\end{cor}
\begin{proof}
By Example \ref{egtriangular}, we know that there is a completely isometric homomorphism\linebreak $\Phi:~\theta(\cA)~\to~\cL_{\bf{}k}$, for some net $\mathbf{k}:\Lambda\to \bbN$. Thus, $\Phi\circ \theta$ is completely bounded by virtue of Corollary~\ref{cor4.15}. Since $\Phi$ is completely isometric, we conclude that $\theta$ is completely bounded as well.
\end{proof}

In closing, we exhibit an example showing that the total reduction property cannot simply be removed from the assumptions of Corollary \ref{cor4.15}.

\begin{eg}\label{E:trprfd}
For each $n\in \bbN$, let $t_n:\bbM_n\to \bbM_n$ denote the transpose map. Then, it is well-known that $\|t_n\|=1$ and
\[
\|t_n^{(n)}\|=n.
\]
Next, let $\cA_n\subset \bbM_{2n}$ denote the unital subalgebra consisting of elements of the form
\[
\begin{bmatrix}
\lambda I_n & A\\
0 & \lambda I_n
\end{bmatrix}
\]
where $\lambda\in \bbC, A\in \bbM_{n}$. The map $\theta_n:\cA_n\to \bbM_{2n}$ defined as
\[
\theta_n\left(\begin{bmatrix}
\lambda I_n & A\\
0 & \lambda I_n
\end{bmatrix} \right)=\begin{bmatrix}
\lambda I_n & t_n(A)\\
0 & \lambda I_n
\end{bmatrix}, \quad \lambda\in \bbC, A\in \bbM_{2n}
\]
is easily verified to be a unital homomorphism with the property that $\|\theta_n\|=1$ and 
\[
\|\theta_n^{(n)}\|\geq \|t_n^{(n)}\|=n.
\]
Consequently, we see that the map $\Theta:\prod_n \cA_n\to \prod_n \bbM_{2n}$ defined as
\[
\Theta(a_n)_n=(\theta_n(a_n))_n, \quad (a_n)_n\in \prod_n \cA_n
\]
is a unital bounded representation which is not completely bounded. Finally, fix $m\in \bbN$ and note that the subspace $\bbC^m\oplus \{0\}$ is invariant for $\cA_m$ yet it does not admit an invariant topological complement as a straightforward calculation establishes. Therefore, $\cA_m$ does not have the total reduction property. Since $\cA_m$ is a closed two-sided ideal of $\prod_n \cA_n$, we conclude from Theorem \ref{thmideal} that $\prod_n \cA_n$ does not have the total reduction property either.
\qed
\end{eg}
\bibliographystyle{plain}
\bibliography{2015_11papers}

\end{document}